\documentclass{amsart}
\usepackage{fancyhdr}
\usepackage{amsthm}
\usepackage{amssymb}
\usepackage{amscd}
\usepackage{amsmath, pb-diagram}
\usepackage[all]{xy}
\usepackage{setspace}
\usepackage{graphics,epsfig}
\usepackage{graphicx}
\usepackage{caption}
\usepackage{float}
\usepackage{array}
\usepackage[english]{babel}
\usepackage[mathscr]{euscript}
\usepackage{mathtools}
\usepackage{url}
\usepackage{enumitem}
\usepackage{hyperref}
\usepackage{tikz}
\usetikzlibrary{arrows.meta}

\hypersetup{colorlinks=true, allcolors=blue}

\allowdisplaybreaks

\DeclareMathOperator{\lcm}{lcm}

\newcommand{\tp}{\textbf{p}}
\newcommand{\ttp}{\textbf{\emph{p}}}
\newcommand{\LL}{\mathscr{L}}
\newcommand{\squad}{\hspace{0.5em}}

\theoremstyle{plain}
\newtheorem{theorem}{Theorem}[section]
\newtheorem{corollary}[theorem]{Corollary}
\newtheorem{lemma}[theorem]{Lemma}
\newtheorem{proposition}[theorem]{Proposition}
\newtheorem{problem}[theorem]{Problem}

\theoremstyle{definition}
\newtheorem{defn}[theorem]{Definition}
\newtheorem{example}[theorem]{Example}
\newtheorem{remark}[theorem]{Remark}
\newtheorem{algorithm}[theorem]{Algorithm}
\newtheorem{question}[theorem]{Question}

\title{Basis condition for generalized spline modules}

\author[S. F\.{I}\c{s}ekc\.{I}]{Seher F\.{I}\c{s}ekc\.{I}}
\address{TED University}
\email{seher.fisekci@tedu.edu.tr}

\author[S. Sar{\i}o\u{g}lan]{Samet Sar{\i}o\u{g}lan}
\address{Hacettepe University}
\email[Corresponding author.]{ssarioglan@hacettepe.edu.tr}

\subjclass[2010]{05C78, 15A15}
\keywords{generalized spline, module, basis} 

\begin{document}
\begin{abstract}
A generalized spline on an edge labeled graph $(G,\alpha)$ is defined as a vertex labeling, such that the difference of labels on adjacent vertices lies in the ideal generated by the edge label.  We study generalized splines over greatest common divisor domains and present a determinantal basis condition for generalized spline modules on arbitrary graphs. The main result of the paper answers a conjecture that appeared in several papers.
\end{abstract}

\maketitle

%%%%%%%%%%%%%%%%%%%%%%%%%%%%%%%%%%%%%%%%%%%%%%%%%%%%%%%%%
%%%%%%%%%% Introduction Section
%%%%%%%%%%%%%%%%%%%%%%%%%%%%%%%%%%%%%%%%%%%%%%%%%%%%%%%%%

\section{Introduction}\label{intro}

Splines are known as piecewise polynomial functions defined on faces of polyhedral complexes with a smoothness condition.  They control the curvature of a surface in engineering and industry.  They also appear in numerical analysis, optimization theory, computer-aided design and modeling, and solutions of differential equations.

The notion of a generalized spline is introduced by Gilbert, Viel, and Tymoczko \cite{GPT}. As a generalization of classical splines, generalized splines are defined on edge labeled graphs over arbitrary rings instead of polyhedral complexes over polynomial rings. The set of all generalized splines on an edge labeled graph over a base ring $R$ has a ring and an $R$-module structure~\cite{GPT}. Module structure of $R_{(G,\alpha)}$ has been studied chiefly in terms of freeness, minimum generating sets, and bases~\cite{HMR, BT, BHKR, PSTW, Alt3, AD, AMT, DM, RS}. Furthermore, generalized splines have also been viewed in terms of homological algebraic methods and with combinatorial, graph-theoretic approaches~\cite{Dip, And, Alt5}.

In this paper, we focus on the following question about generalized spline modules over greatest common divisor (GCD) domains. A GCD domain $R$ is an integral domain such that any two elements of $R$ have a greatest common divisor.

\begin{question} When the base ring $R$ is a GCD domain, under what conditions does a given set of generalized splines form an $R$-module basis for generalized spline modules?
\label{1que1}
\end{question}

Question~\ref{1que1} has been investigated in many papers using determinant-based methods for particular types of graphs and base rings. Gjoni studied generalized splines over integers and gave a basis criterion on cycles~\cite{Gjo}. A similar technique is used for generalized splines over integers on diamond graphs~\cite{Mah,Bla}. Alt{\i}nok and Sar{\i}o\u{g}lan extended these results to GCD domains for cycles, diamond graphs, and trees~\cite{Alt2}. Calta and Rose presented a basis criterion for $R_{(G,\alpha)}$ on arbitrary graphs over GCD domains under certain conditions on edge labels~\cite{RC}.

This paper aims to present a determinantal basis condition for generalized spline modules on arbitrary graphs over GCD domains. The final result of the paper, Theorem~\ref{5thm1}, was stated as a conjecture in two papers~\cite{Alt2, RC}. We fill this gap in the literature by proving it and finalize the problem. In Section~\ref{pre}, we give fundamental definitions and fix notations. We support these new notations with an example. In Section~\ref{dm}, we explain the determinantal methods. We prove Theorem~\ref{3thm1} for arbitrary graphs, which is stated for special types of graphs in previous works. In Section~\ref{bcocg}, we focus on generalized splines on complete graphs over GCD domains and introduce an algorithm that produces special generalized splines. We illustrate the algorithm with an example. We then present a basis condition that answers Question~\ref{1que1} for complete graphs. In Section~\ref{bcoag}, we extend our outcomes to arbitrary graphs. We first define the completion of an edge labeled graph, then clarify Question~\ref{1que1} for arbitrary graphs.

%%%%%%%%%%%%%%%%%%%%%%%%%%%%%%%%%%%%%%%%%%%%%%%%%%%%%%%%%
%%%%%%%%%% Preliminaries Section
%%%%%%%%%%%%%%%%%%%%%%%%%%%%%%%%%%%%%%%%%%%%%%%%%%%%%%%%%

\section{Preliminaries}\label{pre}

We begin with the definition of a generalized spline.

\begin{defn} Given a finite graph $G=(V, E)$ with $n$ vertices and a commutative ring $R$ with identity, an edge labeling on $G$ is defined as a function $\alpha: E \to I(R)$ where $I(R)$ is the set of ideals of $R$.  A generalized spline on $(G,\alpha)$ is a vertex labeling $F = \big( F(v_1) , F(v_2) , \ldots , F(v_n) \big) \in R^n$ such that for each edge $uv$, the difference $F(u) - F(v) \in \alpha(uv)$.
\end{defn}

In order to specify the graph itself, we use $V(G)$ and $E(G)$ for the vertex set and edge set of $G$.

Let $R_{(G,\alpha)}$ denote the set of all generalized splines on $(G,\alpha)$ with base ring $R$.  Then, $R_{(G,\alpha)}$ has a ring and an $R$-module structure by componentwise addition and multiplication by the elements of $R$.  From now on, we refer to generalized splines as splines. We use principal ideals for edge labels and denote each edge label by the generator of the ideal. We use the notations $e_j$ or $uv$ for an edge of $G$, and symbolize the corresponding edge label by $\alpha(e_j) = l_j$ and $\alpha(uv) = l_{uv}$. We also refer to generators of edge labels as edge labels. Figure~\ref{fig1} illustrates an edge labeled diamond graph with base ring $\mathbb{Z}$.

\begin{figure}[H]
\begin{center}
\begin{tikzpicture}
\begin{scope}[every node/.style={circle,thick,draw}]
    \node (A) at (1.5,3) {$v_1$};
    \node (B) at (1.5,0) {$v_2$};
    \node (C) at (3,1.5) {$v_3$};
    \node (D) at (0,1.5) {$v_4$};
\end{scope}
\begin{scope}[>={Stealth[black]},
              every node/.style={fill=white,circle},
              every edge/.style={draw=black,very thick}]
    \path (A) edge node {$5$} (B);
    \path (A) edge node {$4$} (C);
    \path (A) edge node {$6$} (D);          
    \path (B) edge node {$2$} (C);
    \path (B) edge node {$9$} (D);
\end{scope}
\end{tikzpicture}
\caption{Edge labeled diamond graph}
\label{fig1}
\end{center}
\end{figure}
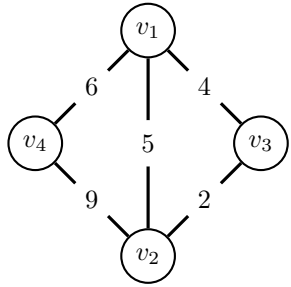

A spline on Figure~\ref{fig1} can be given by $F = (2,32,34,50)$. 

This paper focuses on splines over GCD domains. We shorten the notation $\gcd(a,b)$ and $\lcm(a,b)$ as $(a,b)$ and $[a,b]$. We denote the greatest common divisor and least common multiple of elements of a set $A$ by $\big( A \big)$ and $\big[ A \big]$. Throughout the paper, we assume that $R$ is a GCD domain unless otherwise stated.

Given non-empty subsets $A_1 , A_2 , \ldots , A_k$ of a GCD domain, we define the product by
\begin{displaymath}
A_1 \cdot A_2 \cdots A_k = \{ a_1 \cdot a_2 \cdots a_k \text{ $\vert$ } a_i \in A_i, 1 \leq i \leq k\}.
\end{displaymath}
The following lemma can be proved by combining elements and using the greatest common divisor and least common multiple properties.

\begin{lemma} Let $A_1 , A_2 , \ldots , A_k$ be non-empty subsets of a GCD domain. Then,
\begin{enumerate}[label=\emph{(\alph*)}]
	\item $\big( A_1 \cdot A_2 \cdots A_k \big) = \big( A_1 \big) \cdot \big( A_2 \big) \cdots \big( A_k \big)$.
	\item $\big[ A_1 \cdot A_2 \cdots A_k \big] = \big[ A_1 \big] \cdot \big[ A_2 \big] \cdots \big[ A_k \big]$.	
\end{enumerate}
\label{2lem1}
\end{lemma}

A special type of splines, called flow-up class, is defined as follows.

\begin{defn} Given an edge labeled graph $(G,\alpha)$, an $i$-th flow-up class is defined as a spline $F^{(i)} \in R_{(G,\alpha)}$ such that $F^{(i)} (v_j) = 0$ for $j < i$ and $F^{(i)} (v_i) \neq 0$. 
\end{defn}

For instance, the flow-up classes in Figure~\ref{fig1} can be listed as $F^{(1)} = (1,1,1,1)$, $F^{(2)} = (0,30,0,48)$, $F^{(3)} = (0,0,8,0)$ and $F^{(4)} = (0,0,0,36)$. 

Let $(G,\alpha)$ be an edge labeled graph and $v_i , v_j \in V(G)$. A $v_j$-trail of $v_i$, denoted by $\tp ^{(v_i,v_j)}$, is a sequence of vertices and edges starting at $v_i$ and ending at $v_j$ such that no edge is repeated. We represent a trail by the sequence of the edge labels on it. The length of a trail, denoted by $\big\vert \tp ^{(v_i,v_j)} \big\vert$, is the number of edges it contains. We symbolize the greatest common divisor of edge labels on $\tp ^{(v_i , v_j)}$ by $\big( \tp ^{(v_i , v_j)} \big)$.

Given an edge labeled graph $(G,\alpha)$, fix a vertex $v_i$ with $i > 1$ and let the vertices $v_j$ with $j < i$ be labeled by zero. In this case, $\tp ^{(v_i,v_j)}$ is called a zero trail of $v_i$. Under this setting, note that each zero trail of $v_i$ corresponds to a $v_j$-trail of $v_i$ with $j < i$. An arbitrary zero trail of $v_i$ is denoted by $\tp ^{(v_i , 0)}$. A zero trail with length $1$ is called a zero edge. Let the zero trails of $v_i$ be given by the set $\Big\{ \tp_t ^{(v_i , 0)} \text{ $\vert$ } t=1,2,\ldots,k \Big\}$. We define the element $\LL_i \in R$ as
\begin{displaymath}
\LL_i = \Big[ \big( \tp_1 ^{(v_i , 0)} \big) , \big( \tp_2 ^{(v_i , 0)} \big) , \ldots , \big( \tp_k ^{(v_i , 0)} \big) \Big]
\end{displaymath}
and we set $\LL_1 = 1$. Over principal ideal domains, there exist flow-up classes $F^{(i)} \in R_{(G,\alpha)}$ for each $i$ with $F^{(i)} (v_i) = \LL_i$, and such flow-up classes form an $R$-module basis for $R_{(G,\alpha)}$~\cite{Alt1}. However, such flow-up classes may not exist if $R$ is not a principal ideal domain.

Our first observation on trails is given below.

\begin{theorem} Let $v_i , v_j \in V(G)$ and $\ttp ^{(v_i,v_j)}$ be a $v_j$-trail of $v_i$. If $F \in R_{(G,\alpha)}$, then $\big( \ttp ^{(v_i,v_j)} \big)$ divides $F(v_i) - F(v_j)$.
\label{2thm1}
\end{theorem}

\begin{proof} Assume that $\tp ^{(v_i,v_j)} = < l_{v_i u_1} , l_{u_1 u_2} , \ldots , l_{u_{k-1} u_k} , l_{u_k v_j} >$. Since $F \in R_{(G,\alpha)}$, we have
\begin{align*}
F(v_i) - F(u_1) &= r_1 l_{v_i u_1} \\
F(u_1) - F(u_2) &= r_2 l_{u_1 u_2} \\
&\vdots \\
F(u_{k-1}) - F(u_k) &= r_k l_{u_{k-1} u_k} \\
F(u_k) - F(v_j) &= r_{k+1} l_{u_k v_j}
\end{align*}
for some $r_1 , r_2 , \ldots , r_{k+1} \in R$. Hence, we write
\begin{displaymath}
F(v_i) - F(v_j) = r_1 l_{v_i u_1} + r_2 l_{u_1 u_2} + \ldots + r_k l_{u_{k-1} u_k} + r_{k+1} l_{u_k v_j}.
\end{displaymath}
The greatest common divisor $\big( \tp ^{(v_i,v_j)} \big)$ divides each summand, and thus, it divides $F(v_i) - F(v_j)$.
\end{proof}

We can simplify the computation of $\LL_i$ as below.

\begin{remark} Let $\tp_j ^{(v_i , 0)}$ and $\tp_k ^{(v_i , 0)}$ be two zero trails of $v_i$ with $\tp_j ^{(v_i , 0)} \subset \tp_k ^{(v_i , 0)}$. In this case, it is sufficient to consider just $\tp_j ^{(v_i , 0)}$ to compute $\LL_i$ since $\big( \tp_k ^{(v_i , 0)} \big)$ divides $\big( \tp_j ^{(v_i , 0)} \big)$. Therefore, we consider only zero trails of $v_i$ that do not contain any other.
\label{2rem1}
\end{remark}

Let $\tp_t ^{(v_i , 0)} = < l_{t_1}, l_{t_2}, \ldots, l_{t_s} >$ be a zero trail of $v_i$ with $t_s \neq 1$. We write $l_{t_j} = \big( \tp_t ^{(v_i , 0)} \big) \cdot {l_{t_j} ^{(t)}} '$ for $1 \leq j \leq s$ where ${l_{t_j} ^{(t)}}' \in R$ such that $\Big( {l_{t_1} ^{(t)}}' , {l_{t_2} ^{(t)}}' , \ldots, {l_{t_s} ^{(t)}}'  \Big) = 1$ and use the notations 
\begin{align*}
D^i _t &= \Big\{ {l_{t_1} ^{(t)}}' , {l_{t_2} ^{(t)}}' , \ldots, {l_{t_s} ^{(t)}}' \Big\} \text{ and} \\ 
\mathfrak{D}_i &= \{ D^i _1 , D^i _2, \ldots , D^i _{k_i} \}
\end{align*}
where $k_i$ is the number of zero trails of $v_i$ with length greater than $1$. Here $\big( D^i _t \big) = 1$ for all $t$, and hence, $\big( \mathfrak{D}_i \big) = \Big( \big( D^i _1 \big) , \big( D^i _2 \big) , \ldots , \big( D^i _{k_i} \big) \Big) = 1$ for all $2 \leq i \leq n-1$. We define the following set similar to the cartesian product, whose elements are subsets instead of tuples
\begin{displaymath}
\bigtimes D^i = \Big\{  \big\{ {l_{j_1} ^{(1)}}' , {l_{j_2} ^{(2)}}' , \ldots, {l_{j_{k_i}} ^{(k_i)}}' \big\} \text{ $\vert$ } {l_{j_t} ^{(t)}}' \in D^i _t , \squad 1 \leq t \leq k_i \Big\}
\end{displaymath}
and symbolize an arbitrary element of $\bigtimes D^i$ by $a ^{(i)}$.

\begin{remark} When $i = n$, all vertices except $v_i$ are labeled by zero and each zero trail of $v_i$ contains a zero edge. In this case, we only consider zero edges by Remark~\ref{2rem1} and hence, $\mathfrak{D}_n = \emptyset$.
\label{2rem2} 
\end{remark}

For each element $a ^{(i)} = \Big\{ {l_{j_1} ^{(1)}}' , {l_{j_2} ^{(2)}}' , \ldots, {l_{j_{k_i}} ^{(k_i)}}' \Big\} \in \bigtimes D^i$, we say that the edge label $l_{j_s}$ lies in $a ^{(i)}$ for $1 \leq s \leq k_i$. An edge label from each zero trail $\tp ^{(v_i , 0)}$ with $\big\vert \tp ^{(v_i , 0)} \big\vert > 1$ is contained in $l(a ^{(i)})$ by construction. The set of all edge labels that lie in $a ^{(i)}$ is denoted by $l(a ^{(i)})$. We use the notation $\prod a ^{(i)}$ for the product of elements of $a ^{(i)}$.

A subgraph $H$ of $G$ is a graph such that $V(H) \subset V(G)$ and $E(H) \subset E(G)$. For each element $a ^{(i)} \in \bigtimes D^i$, we define the subgraph $H_{a ^{(i)}} \subset G$ with edge set $E(H_{a ^{(i)}}) = \{ e_j \text{ $\vert$} l_j \in l(a ^{(i)}) \}$. Note that $v_i v_j \not\in E(H_{a ^{(i)}})$ for $j<i$ by the definition of $a ^{(i)}$, since $< l_{v_i v_j} >$ is a zero edge of $v_i$. 

\begin{example} Consider the edge labeled $K_4$ in Figure~\ref{fig2}.

\begin{figure}[!htb]
\centering
\begin{minipage}{.5\textwidth}
\centering
\begin{tikzpicture}
\begin{scope}[every node/.style={circle,thick,draw}]
    \node (A) at (2,3.5) {$v_1$};
    \node (B) at (0,0) {$v_2$};
    \node (C) at (4,0) {$v_3$};
    \node (D) at (2,1.5) {$v_4$};
\end{scope}
\begin{scope}[>={Stealth[black]},
              every node/.style={fill=white,circle},
              every edge/.style={draw=black,very thick}]
    \path (A) edge[bend right=15] node {$l_1$} (B);
    \path (A) edge[bend left=15] node {$l_3$} (C);
    \path (A) edge node {$l_4$} (D);         
    \path (B) edge node {$l_2$} (C);
    \path (B) edge node {$l_5$} (D);
    \path (C) edge node {$l_6$} (D);
\end{scope}
\end{tikzpicture}
\caption{\small{Edge labeled $K_4$}}
\label{fig2}
\end{minipage}%
\begin{minipage}{.5\textwidth}
\centering
\begin{tikzpicture}
\begin{scope}[every node/.style={circle,thick,draw}]
    \node (A) at (2,3.5) {$v_1$};
    \node (B) at (0,0) {$v_2$};
    \node (C) at (4,0) {$v_3$};
    \node (D) at (2,1.5) {$v_4$};
\end{scope}
\begin{scope}[>={Stealth[black]},
              every node/.style={fill=white,circle},
              every edge/.style={draw=black,very thick}]
    \path (A) edge node {$l_4$} (D);         
    \path (B) edge node {$l_2$} (C);
    \path (C) edge node {$l_6$} (D);
\end{scope}
\end{tikzpicture}
\caption{\small{The subgraph $H_a$}}
\label{fig3}
\end{minipage}
\end{figure}

\noindent The zero trails of $v_2$ are determined as \squad $\tp ^{(v_2 , 0)} _1 = <l_1>$, \squad $\tp ^{(v_2 , 0)} _2 = <l_2 , l_3>$, \squad $\tp ^{(v_2 , 0)} _3 = <l_5 , l_4>$, \squad $\tp ^{(v_2 , 0)} _4 = <l_2 , l_6 , l_4>$ \squad and \squad $\tp ^{(v_2 , 0)} _5 = <l_5 , l_6 , l_3>$. Thus,
\begin{displaymath}
\mathfrak{D}_2 = \Big\{ \{ {l_2 ^{(2)}} ' , {l_3 ^{(2)}} ' \}, \quad \{ {l_5 ^{(3)}} ' , {l_4 ^{(3)}} ' \}, \quad \{ {l_2 ^{(4)}} ' , {l_6 ^{(4)}} ' , {l_4 ^{(4)}} ' \}, \quad \{ {l_5 ^{(5)}} ' , {l_6 ^{(5)}} ' , {l_3 ^{(5)}} ' \} \Big\}.
\end{displaymath}
Here $a ^{(2)} = \Big\{ {l_2 ^{(2)}} ' , {l_4 ^{(3)}} ' , {l_6 ^{(4)}} ' , {l_6 ^{(5)}} ' \Big\} \in \bigtimes D^2$ and the subgraph $H_{a ^{(2)}} \subset K_4$ is given in Figure~\ref{fig3}. Moreover, $l (a ^{(2)}) = \{ l_2 , l_4 , l_6 \}$ and $\prod a ^{(2)} = {l_2 ^{(2)}} ' \cdot {l_4 ^{(3)}} ' \cdot {l_6 ^{(4)}} ' \cdot {l_6 ^{(5)}} '$.
\label{2ex1}
\end{example}

\begin{remark} The subgraph $H_{a ^{(i)}}$ contains at least one edge from each zero trail $\tp ^{(v_i , 0)}$ with $\big\vert \tp ^{(v_i , 0)} \big\vert > 1$ by definition of the element $a ^{(i)}$.
\label{2rem3}
\end{remark}

As a result of Lemma~\ref{2lem1}, we present the following proposition.

\begin{proposition} If $\mathbb{A}$ is the set of products $\prod a ^{(2)} \cdot \prod a ^{(3)} \cdots \prod a ^{(n-1)}$ where $a ^{(i)} \in \bigtimes D^i$ are arbitrary elements for $2 \leq i \leq n-1$, then we have
\begin{center}
$\big( \mathbb{A} \big) = \Big( \big( \mathfrak{D}_2 \big) , \big( \mathfrak{D}_3 \big) , \ldots , \big( \mathfrak{D}_{n-1} \big) \Big) = 1$.
\end{center}
\label{2prop1}
\end{proposition}

Let $a_1 ^{(i)}, a_2 ^{(i)} \in \bigtimes D^i$. We say that $a_1 ^{(i)} \subset a_2 ^{(i)}$ if $l( a_1 ^{(i)}) \subset l(a_2 ^{(i)})$. If there is no element $a_2 ^{(i)}$ with $a_2 ^{(i)} \subsetneqq a_1 ^{(i)}$, we call $a_1 ^{(i)}$ a minimal element. When $a ^{(i)}$ is minimal, deleting an arbitrary edge from $H_{a ^{(i)}}$ yields a zero trail on $(G,\alpha)$ with no edge it contains lies in $a ^{(i)}$. The minimal elements of $\bigtimes D^i$ in Figure~\ref{fig2} are listed in terms of containing edge labels by \squad $l (a_1 ^{(2)}) = \{ l_2 , l_5 \}$, \squad $l (a_2 ^{(2)}) = \{ l_3 , l_4 \}$, \squad $l (a_3 ^{(2)}) = \{ l_2 , l_4 , l_6 \}$ \squad and \squad $l (a_4 ^{(2)}) = \{ l_3 , l_5 , l_6 \}$.

The following proposition plays a crucial role in the proof of the main results of the paper.

\begin{proposition} Let $a ^{(i)} \in \bigtimes D^i$ and ${l_{t_j} ^{(t)}}' \in a ^{(i)}$ be the component that comes from the zero trail $\ttp_t ^{(v_i , 0)} = < l_{t_1}, l_{t_2}, \ldots, l_{t_s} >$ of $v_i$. Then, the edge label $l_{t_j}$ divides $\prod a ^{(i)} \cdot \LL_i$.
\label{2prop2}
\end{proposition}

\begin{proof} Recall that $l_{t_j} = \big( \tp_t ^{(v_i , 0)} \big) \cdot {l_{t_j} ^{(t)}} '$. Here ${l_{t_j} ^{(t)}}'$ divides the product $\prod a ^{(i)}$ and $\big( \tp_t ^{(v_i , 0)} \big)$ divides $\LL_i$. Thus, $\big( \tp_t ^{(v_i , 0)} \big) \cdot {l_{t_j} ^{(t)}} ' = l_{t_j}$ divides $\prod a ^{(i)} \cdot \LL_i$.
\end{proof}

In addition to Proposition~\ref{2prop2}, the label of any zero edge of $v_i$ also divides $\prod a ^{(i)} \cdot \LL_i$.

\begin{proposition} Given an edge labeled $(G , \alpha)$, a zero edge $e_j$ of $v_i$, and $a ^{(i)} \in \bigtimes D^i$, the edge label $l_j$ divides $\prod a ^{(i)} \cdot \LL_i$.
\label{2prop3}
\end{proposition}

\begin{proof} Since $e_j$ is a zero edge of $v_i$, the edge label $l_j$ divides $\LL_i$, and so $\prod a ^{(i)} \cdot \LL_i$.
\end{proof}

%%%%%%%%%%%%%%%%%%%%%%%%%%%%%%%%%%%%%%%%%%%%%%%%%%%%%%%%%
%%%%%%%%%% Determinantal Methods
%%%%%%%%%%%%%%%%%%%%%%%%%%%%%%%%%%%%%%%%%%%%%%%%%%%%%%%%%

\section{Determinantal Methods}\label{dm}

Let $(G, \alpha)$ be an edge labeled graph with $n$-vertices and $A = \{ F_1 , \ldots , F_n \} \subset R_{(G,\alpha)}$. We can rewrite $A$ in a matrix form, whose columns are the elements of $A$, such as
\begin{displaymath}
A = \begin{pmatrix} F_1 (v_n) & F_2 (v_n) & \ldots & F_n (v_n) \\ \vdots \\ F_1 (v_2) & F_2 (v_2) & \ldots & F_n (v_2) \\ F_1 (v_1) & F_2 (v_1) & \ldots & F_n (v_1) \end{pmatrix}.
\end{displaymath}
$A$ is called a spline matrix and the determinant $\big\vert A \big\vert$ is denoted by $\big\vert F_1 \squad F_2 \squad \ldots \squad F_n \big\vert$. 

Given an edge labeled graph $(G,\alpha)$ with $n$ vertices, we present the element $Q_G$ as
\begin{displaymath}
Q_G = \prod\limits_{i=1} ^n \LL_i.
\end{displaymath}

We give the basis condition for $R_{(G,\alpha)}$ by using $Q_G$ and the determinant $\big\vert A \big\vert$. An explicit formula for $Q_G$ is given in terms of edge labels for cycles, diamond graphs, and trees~\cite{Gjo, Mah, Bla, Alt2}. It is difficult to determine such a formula when the number of edges of $G$ is greater than the number of vertices. We reach our goal without providing a general formula for $Q_G$ for arbitrary graphs. 

The properties of the determinant $\big\vert A \big\vert$ are listed below.

\begin{proposition} Let $(G, \alpha)$ be an edge labeled graph with $n$-vertices. Assume that $\{ F_1 , \ldots , F_n \}$ forms a basis for $R_{(G,\alpha)}$ and let $\{ G_1 , \ldots , G_n \} \subset R_{(G,\alpha)}$. Then $\big\vert F_1 \squad F_2 \squad \ldots \squad F_n \big\vert$ divides $\big\vert G_1 \squad G_2 \squad \ldots \squad G_n \big\vert$.
\label{3prop1}
\end{proposition}

\begin{proof} See Lemma 5.1.4. in~\cite{Gjo}.
\end{proof}

\begin{proposition} Let $(G, \alpha)$ be an edge labeled graph with $n$-vertices. Let\linebreak $\{ F_1 , F_2, \ldots , F_n \} \subset R_{(G , \alpha)}$. If $\big\vert F_1 \squad F_2 \squad \ldots \squad F_n \big\vert = r \cdot Q_G$ where $r \in R$ is a unit, then $\{ F_1 , F_2 , \ldots , F_n \}$ forms a basis for $R_{(G , \alpha)}$.
\label{3prop2}
\end{proposition}

\begin{proof} The proof of Lemma 3.19 in~\cite{Alt2} holds for an arbitrary edge labeled graph $(G,\alpha)$.
\end{proof}

\begin{theorem} Given an edge labeled graph $(G, \alpha)$ with $n$-vertices and a set of splines $\{ F_1 , F_2, \ldots , F_n \} \subset R_{(G , \alpha)}$, the element $Q_G$ divides $\big\vert F_1 \squad F_2 \squad \ldots \squad F_n \big\vert$. 
\label{3thm1}
\end{theorem}

\begin{proof} Let $F_i = \big( F_i (v_1) , \ldots , F_i (v_n) \big)$ and $\tp ^{(v_2 , 0)} , \tp ^{(v_3 , 0)} , \ldots , \tp ^{(v_n , 0)}$ be arbitrary zero trails of $v_i$ for $2 \leq i \leq n$. Then, $\tp ^{(v_i , 0)}$ corresponds to a $v_{i_j}$-trail of $v_i$ with $i_j < i$. By using elementray row operations on the spline matrix $A$, we replace the $(n-i+1)$-th row of $A$ as
\begin{displaymath}
\big( F_1 (v_i) \squad , \ldots , \squad F_n (v_i) \big) \to \big( F_1 (v_i) - F_1 (v_{i_j}) \squad, \ldots , \squad F_n (v_i) - F_n (v_{i_j}) \big).
\end{displaymath}
Here $\big( \tp ^{(v_i , 0)} \big)$ divides this new row by Theorem~\ref{2thm1}. Note that such row operations fix the last row of $A$. After these proper row operations, one concludes that the product $\big( \tp ^{(v_2 , 0)} \big) \cdot \big( \tp ^{(v_3 , 0)} \big) \cdots \big( \tp ^{(v_n , 0)} \big)$ divides the determinant $\big\vert A \big\vert$. Since the zero trails are chosen arbitrarily, we conclude that each element of the set
\begin{displaymath}
\mathbb{P} = \Big\{ \big( \tp ^{(v_2 , 0)} \big) \cdot \big( \tp ^{(v_3 , 0)} \big) \cdots \big( \tp ^{(v_n , 0)} \big) \text{ $\vert$ } \tp ^{(v_i , 0)} \text{ is a zero trail of $v_i$ for $2 \leq i \leq n$} \Big\}
\end{displaymath}
divides $\big\vert A \big\vert$. Therefore, the least common multiple
\begin{displaymath}
\big[ \mathbb{P} \big] = \LL_1 \cdot \LL_2 \cdot \LL_3 \cdots \LL_n = Q_G
\end{displaymath}
divides $\big\vert A \big\vert = \big\vert F_1 \squad F_2 \squad \ldots \squad F_n \big\vert$ by Lemma~\ref{2lem1}.
\end{proof}

%%%%%%%%%%%%%%%%%%%%%%%%%%%%%%%%%%%%%%%%%%%%%%%%%%%%%%%%%
%%%%%%%%%% Basis Condition on Complete Graphs Section
%%%%%%%%%%%%%%%%%%%%%%%%%%%%%%%%%%%%%%%%%%%%%%%%%%%%%%%%%

\section{Basis Condition on Complete Graphs}\label{bcocg}

A complete graph with $n$ vertices, denoted by $K_n$, is a graph such that $uv \in E(K_n)$ for each distinct pair of vertices $u,v \in V(K_n)$. The most challenging graph type of studying splines is complete graphs because any pair of vertices is connected, and there are many spline conditions to check. Therefore, previously, determinantal techniques have never been used efficiently for module bases of $R_{(K_n , \alpha)}$. In this section, we focus on splines on complete graphs over GCD domains and give a determinantal basis criterion for $R_{(K_n , \alpha)}$.

The first aim of this section is to answer the following problem.

\begin{problem} Let $(K_n , \alpha)$ be an edge labeled complete graph and $R$ be a GCD domain. Fix an integer $i$ with $2 \leq i \leq n-1$. For each element of $a ^{(i)} \in \bigtimes D^i$, can we find a spline $F \in R_{(K_n , \alpha)}$ such that $F (v) \in \Big\{ 0, \squad \prod a ^{(i)} \cdot \LL_i \Big\}$ for all $v \in V(K_n)$? We symbolize such a spline by $F_{a ^{(i)}}$.
\label{4prob1}
\end{problem}

Problem~\ref{4prob1} is easy to answer for any type of graph in the following case.

\begin{proposition} Let $(G,\alpha)$ be an arbitrary edge labeled graph and assume that $l_{u v_i}$ lie in $a ^{(i)} \in \bigtimes D^i$ for all $u v_i \in E(G)$. A vertex labeling $g_{a ^{(i)}}$ such that $g_{a ^{(i)}} (v_i) = \prod a ^{(i)} \cdot \LL_i$ and $g_{a ^{(i)}} (v) = 0$ for $v \neq v_i$ is a spline on $R_{(G , \alpha)}$.
\label{4prop2}
\end{proposition}

\begin{proof} Consider the edge $uv \in E(G)$. Assume that $u$ or $v$ is equal to $v_i$, then $F(u) - F(v) = \pm \prod a ^{(i)} \cdot \LL_i$ and $l_{uv}$ divides the difference by Proposition~\ref{2prop2}. Hence, the spline condition holds on $uv$. If none of $u$ or $v$ is equal to $v_i$, then $g_{a ^{(i)}} (u) = g_{a ^{(i)}} (v) = 0$ and the spline condition holds on $uv$. Thus, $g_{a ^{(i)}} \in R_{(G, \alpha)}$.
\end{proof}

We consider the case of $l_{u v_i}$ do not lie in $a ^{(i)}$ for some $u v_i \in E(K_n)$ in Algorithm~\ref{4alg1}. Nevertheless, we first reduce Problem~\ref{4prob1} to minimal elements of $\bigtimes D^i$.

\begin{theorem} Let $a ^{(i)} \in \bigtimes D^i$ be minimal. For any element ${a^*} ^{(i)} \in \bigtimes D^i$ with $a ^{(i)} \subset {a^*} ^{(i)}$, a spline $F_{a ^{(i)}} \in R_{(K_n , \alpha)}$ induces a spline $F_{{a^*} ^{(i)}} \in R_{(K_n , \alpha)}$.
\label{4thm1}
\end{theorem}

\begin{proof} Since $a ^{(i)} \subset {a^*} ^{(i)}$, we have $H_{a ^{(i)}} \subset H_{{a^*} ^{(i)}}$ and $\prod a ^{(i)}$ divides $\prod {a^*} ^{(i)}$. We define the induced spline $F_{{a^*} ^{(i)}}$ as follows.
\begin{displaymath}
F_{{a^*} ^{(i)}} (v) = \begin{cases}
\prod {a^*} ^{(i)} \cdot \LL_i & \text{ if } F_{a ^{(i)}} (v) = \prod a ^{(i)} \cdot \LL_i, \\
0 & \text{ if } F_{a ^{(i)}} (v) = 0.
\end{cases}
\end{displaymath}
To conclude $F_{{a^*} ^{(i)}} \in R_{(K_n , \alpha)}$, consider an arbitrary edge $uv \in E(K_n)$. If $uv \not\in E(H_{{a^*} ^{(i)}})$, then $uv \not\in E(H_{a ^{(i)}})$ and $F_{{a^*} ^{(i)}} (u) = F_{{a^*} ^{(i)}} (v)$ since $F_{a ^{(i)}} \in R_{(K_n , \alpha)}$. So the spline condition holds. If $uv \in E(H_{{a^*} ^{(i)}})$, we have $\alpha(uv)$ divides $\prod {a^*} ^{(i)} \cdot \LL_i$ by Proposition~\ref{2prop2} and the spline condition holds again. Therefore, $F_{{a^*} ^{(i)}} \in R_{(K_n , \alpha)}$.
\end{proof}

We present an application of Theorem~\ref{4thm1} at the end of Example~\ref{4ex1}. We introduce the following algorithm to construct a spline $F_{a ^{(i)}} \in R_{(K_n , \alpha)}$ when $R$ is a GCD domain.

\begin{algorithm} Let $(K_n , \alpha)$ be an edge labeled complete graph and $a ^{(i)} \in \bigtimes D^i$ be a minimal element. Assume that  $v_i v_{k_1} , \ldots , v_i v_{k_t} \not\in E(H_{a ^{(i)}})$ with $k_j > i$ for $1 \leq j \leq t$. Define a vertex labeling $g_{a ^{(i)}}$ on $(K_n , \alpha)$ as follows.
\begin{enumerate}
	\item Set $g_{a ^{(i)}}(v_i) = \prod a ^{(i)} \cdot \LL_i$.
	\item Set $g_{a ^{(i)}}(v_{k_j}) = \prod a ^{(i)} \cdot \LL_i$ for $1 \leq j \leq t$.
	\item For a vertex $v_s$ with $s > i$ and $v_i v_s \in E(H_{a ^{(i)}})$, set $g_{a ^{(i)}}(v_s) = 0$ if $v_s v_{k_j} \in E(H_{a ^{(i)}})$ for $1 \leq j \leq t$. Otherwise, set $g_{a ^{(i)}}(v_s) = \prod a ^{(i)} \cdot \LL_i$.
\end{enumerate}
\label{4alg1}
\end{algorithm}

Recall that any vertex $v_s$ with $s < i$ is pre-labeled by zero by the setting of $a ^{(i)} \in \bigtimes D^i$. Algorithm~\ref{4alg1} agrees these pre-labeled vertices: Assume that $v_s \in V(K_n)$ with $s < i$, then $< l_{v_s v_{k_j}} , l_{v_{k_j} v_i} >$ is a zero trail of $v_i$ and $v_{k_j} v_i \not\in E(H_{a ^{(i)}})$ for $1 \leq j \leq t$. Thus, $v_s v_{k_j} \in E(H_{a ^{(i)}})$ for all $j$ by Remark~\ref{2rem3} and $g_{a ^{(i)}}(v_s) = 0$. 

The results of Algorithm~\ref{4alg1} are presented below.

\begin{proposition} The vertex labeling $g_{a ^{(i)}}$ is a spline in $R_{(H_{a ^{(i)}} , \alpha)}$.
\label{4prop3}
\end{proposition}

\begin{proof} Let $uv \in E(H_{a ^{(i)}})$. If $g_{a ^{(i)}} (u) - g_{a ^{(i)}} (v) = 0$, then the spline condition holds. Assume that $g_{a ^{(i)}} (u) - g_{a ^{(i)}} (v) = \prod a ^{(i)} \cdot \LL_i$. By Proposition~\ref{2prop2}, $l_{uv}$ divides $\prod a ^{(i)} \cdot \LL_i$ and the spline condition holds again. Hence, $g_{a ^{(i)}} \in R_{(H_{a ^{(i)}} , \alpha)}$.
\end{proof}

We state the main result of Algorithm~\ref{4alg1} as follows.

\begin{theorem} The vertex labeling $g_{a ^{(i)}}$ is a spline in $R_{(K_n , \alpha)}$.
\label{4thm2}
\end{theorem}

\begin{proof} Algorithm~\ref{4alg1} covers all vertices $v_j$ with $j > i$ since the graph is complete. Let $uv \in E(K_n)$. If $uv \in E(H_{a ^{(i)}})$, the spline condition holds by Proposition~\ref{4prop3}. Let $uv \not\in E(H_{a ^{(i)}})$ and assume that the spline condition does not hold on $uv$. Without loss of generality, say $g_{a ^{(i)}} (u) = \prod a ^{(i)} \cdot \LL_i$ and $g_{a ^{(i)}} (v) = 0$ where $u = v_s$ and $v = v_{s'}$. Note that $s > i$ because of the label $g_{a ^{(i)}} (u)$. We deal with the rest of the proof in two cases by comparing the indices $s'$ and $i$.

Assume that $s' < i$, then $v v_i \not\in E(H_{a ^{(i)}})$ by the construction of $a ^{(i)}$. If $u v_i \not\in E(H_{a ^{(i)}})$, we have the zero trail of $v_i$ given by $< l_{v_i u} , l_{u v} >$ on which no edge lie in $a ^{(i)}$, and this contradicts to Remark~\ref{2rem3}. If $u v_i \in E(H_{a ^{(i)}})$, we have another vertex $v_k$ with $k > i$ such that $v_i v_k , v_k u \not\in E(H_{a ^{(i)}})$ since $g_{a ^{(i)}} (u) = \prod a ^{(i)} \cdot \LL_i$. In this case, no edge on the zero trail $< l_{v_i v_k}, l_{v_k u} , l_{u v} >$ of $v_i$ lie in $a ^{(i)}$, a contradiction to Remark~\ref{2rem3} again.
	
Suppose that $s' > i$, then $v v_i \in E(H_{a ^{(i)}})$ since $g_{a ^{(i)}} (v) = 0$. Moreover, $u v_i \in E(H_{a ^{(i)}})$ since the other case would contradict the label of $v$. By reason of $u v_i \in E(H_{a ^{(i)}})$ and $g_{a ^{(i)}} (u) = \prod a ^{(i)} \cdot \LL_i$, there exists a vertex $v_k$ with $k > i$ such that $v_k v_i \not\in E(H_{a ^{(i)}})$ and $u v_k \not\in E(H_{a ^{(i)}})$. Since $g_{a ^{(i)}} (v) = 0$, we say $v v_k \in E(H_{a ^{(i)}})$. As $a ^{(i)}$ is minimal, there exists a zero trail $\tp ^{(v_i , 0)}$ such that the only edge on $\tp ^{(v_i , 0)}$ that lie in $a^ {(i)}$ is $v v_k$. However, replacing $l_{v v_k}$ with $l_{v_k u} , l_{uv}$ on $\tp ^{(v_i , 0)}$ yields another zero trail $\tp* ^{(v_i , 0)}$ such that no edge on $\tp* ^{(v_i , 0)}$ lies in $a ^{(i)}$, which contradicts to Remark~\ref{2rem3}. 

Since we have contradiction in all possible cases, the assumption is wrong, and the spline condition holds on $uv$. Thus, $g_{a ^{(i)}} \in R_{(K_n , \alpha)}$.
\end{proof}

We denote the vertex label $g_{a ^{(i)}}$ by $F_{a ^{(i)}}$ in the rest of the paper. We run Algorithm~\ref{4alg1} in the following example.

\begin{example} Consider the edge labeled $(K_5 , \alpha)$ in Figure~\ref{fig6}. A minimal element $a ^{(2)} \in \bigtimes D^i$ is given by $l (a ^{(2)}) = \{ l_2 , l_4 , l_7 , l_5 , l_9 , l_{10} \}$ in terms of edge labels. The subgraph $H_{a ^{(2)}}$ is presented as solid lines, and the other edges are denoted by dashed lines below.

\begin{figure}[H]
\centering
\begin{tikzpicture}
\begin{scope}[every node/.style={circle,thick,draw}]
    \node (A) at (2,4) {$v_1$};
    \node (B) at (0,0) {$v_2$};
    \node (C) at (4,0) {$v_3$};
    \node (D) at (2,1.5) {$v_4$};
    \node (E) at (-1,5) {$v_5$};
\end{scope}
\begin{scope}[>={Stealth[black]},
              every node/.style={fill=white,circle},
              every edge/.style={draw=black,very thick}]
    \path[dashed] (A) edge node {$l_1$} (B);
    \path[dashed] (A) edge node {$l_3$} (C);
    \path (A) edge node {$l_4$} (D);
    \path (A) edge node {$l_7$} (E);          
    \path (B) edge node {$l_2$} (C);
    \path (B) edge node {$l_5$} (D);
    \path[dashed] (B) edge node {$l_8$} (E);
    \path[dashed] (C) edge node {$l_6$} (D);
    \path (C) edge[bend right=75] node {$l_9$} (E);
    \path (D) edge node {$l_{10}$} (E);
\end{scope}
\end{tikzpicture}
\caption{Edge labeled $(K_5 , \alpha)$}
\label{fig6}
\end{figure}
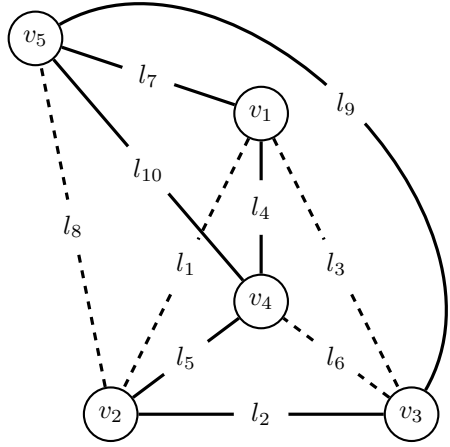

We execute Algorithm~\ref{4alg1} as follows.

\begin{enumerate}
	\item $g_{a ^{(2)}} (v_1) = 0$ as a pre-labeled vertex.
	\item Set $g_{a ^{(2)}} (v_2) = \prod a ^{(2)} \cdot \LL_2$.
	\item Set $g_{a ^{(2)}} (v_5) = \prod a ^{(2)} \cdot \LL_2$ since $v_5 v_2 \not\in E(H_{a ^{(2)}})$.
	\item Set $g_{a ^{(2)}} (v_3) = g_{a ^{(2)}} (v_4) = 0$ since $v_5 v_3 , v_5 v_4 \in E(H_{a ^{(2)}})$.
\end{enumerate}
Therefore, $g_{a ^{(2)}} = F_{a ^{(2)}} = \Big( 0, \prod a ^{(2)} \cdot \LL_2 , 0 , 0 , \prod a ^{(2)} \cdot \LL_2 \Big)$. As an application of Theorem~\ref{4thm1}, $F_{a ^{(2)}}$ induces the spline $F_{{a^*} ^{(2)}} = \Big( 0, \prod {a^*} ^{(2)} \cdot \LL_2 , 0 ,  0 , \prod {a^*} ^{(2)} \cdot \LL_2 \Big)$ for $a ^{(2)} \subset {a^*} ^{(2)}$ with $l ({a^*} ^{(2)}) = \{ l_2 , l_4 , l_7 , l_5 , l_9 , l_{10} , l_3 \}$.
\label{4ex1}
\end{example}

\begin{remark} Since $\mathfrak{D}_n = \emptyset$ by Remark~\ref{2rem2}, we fix $F_{a ^{(n)}} = (0,0, \ldots , \LL_n)$. We conclude that $F_{a ^{(n)}} \in R_{(K_n , \alpha)}$ by Proposition~\ref{2prop3}.
\end{remark}

\begin{remark} The spline $F_{a ^{(i)}} \in R_{(K_n , \alpha)}$ has at least $i$ zero components. To see this, notice that $F_{a ^{(i)}} (v_j) = 0$ for all $j < i$ where $2 \leq i \leq n-1$. In case of $i = n$, we have $F_{a ^{(n)}} = (0,0, \ldots , \LL_n)$.
\label{4rem1}
\end{remark}

We summarize Proposition~\ref{4prop2}, Theorem~\ref{4thm1}, Theorem~\ref{4thm2} and Remark~\ref{4rem1} by the following corollary.

\begin{corollary} Let $(K_n , \alpha)$ be an edge labeled complete graph and $R$ be a GCD domain. For any element $a ^{(i)} \in \bigtimes D^i$ with $i = 2 \leq i \leq n-1$, there corresponds a spline $F_{a ^{(i)}} \in R_{(K_n , \alpha)}$ with at least $i$ zero components.
\label{4cor1}
\end{corollary}

Corollary~\ref{4cor1} answers Problem~\ref{4prob1} positively for complete graphs. The main result of the section is stated below.

\begin{theorem} Let $\{ F_1 , F_2, \ldots , F_n \} \subset R_{(K_n , \alpha)}$. If $\{ F_1 , F_2 , \ldots , F_n \}$ forms a basis for $R_{(K_n , \alpha)}$, then $\big\vert F_1 \squad F_2 \squad \ldots \squad F_n \big\vert = r \cdot Q_{K_n}$ where $r \in R$ is a unit.
\label{4thm3}
\end{theorem}

\begin{proof} Assume that $\{ F_1 , F_2 , \ldots , F_n \}$ forms a basis for $R_{(K_n , \alpha)}$. By Theorem~\ref{3thm1}, $Q_{K_n}$ divides $\big\vert F_1 \squad F_2 \squad \ldots \squad F_n \big\vert$, say $\big\vert F_1 \squad F_2 \squad \ldots \squad F_n \big\vert = r \cdot Q_{K_n}$. We show that $r$ is a unit.

By Corollary~\ref{4cor1}, there exists a spline $F_{a ^{(i)}} \in R_{(K_n , \alpha)}$ with at least $i$ zero components for any $a ^{(i)} \in \bigtimes D^i$ and $2 \leq i \leq n-1$. Construct the spline matrix 
\begin{displaymath}
A = \big( 1 \squad F_{a ^{(2)}} \squad F_{a ^{(3)}} \squad \ldots \squad F_{a ^{(n-1)}} \squad F_{a ^{(n)}} \big)
\end{displaymath}
where $1 = (1,\ldots,1)$ is the trivial spline. The matrix $A$ is an upper triangular matrix since the first $i-1$ components of $F_{a ^{(i)}}$ are zero. Therefore, 
\begin{align*}
\big\vert A \big\vert &= \LL_2 \cdot \prod a ^{(2)} \cdot \LL_3 \cdot \prod a ^{(3)} \cdots \LL_{(n-1)} \cdot \prod a ^{(n-1)} \cdot \LL_n \\ &= \prod a ^{(2)} \cdot \prod a ^{(3)} \cdots \prod a ^{(n-1)} \cdot Q_{K_n}.
\end{align*}
By Proposition~\ref{5prop1}, $\big\vert F_1 \squad F_2 \squad \ldots \squad F_n \big\vert = r \cdot Q_{K_n}$ divides $\big\vert A' \big\vert$ and hence, $r$ divides $\prod a ^{(2)} \cdot \prod a ^{(3)} \cdots \prod a ^{(n-1)}$. Since $a ^{(i)} \in \bigtimes D^i$ are arbitrary, we conclude that $r$ divides each element of the set 
\begin{displaymath}
\mathbb{D} = \Big\{ \prod a ^{(2)} \cdot \prod a ^{(3)} \cdots \prod a ^{(n-1)} \text{ $\vert$ } a ^{(i)} \in \bigtimes D^i , \squad 2 \leq i \leq n-1 \Big\}.
\end{displaymath}
Thus, $r$ divides the greatest common divisor $\big( \mathbb{D} \big)$ which is equal to $1$ by Proposition~\ref{2prop1} and so, $r$ is a unit.
\end{proof}

We give the basis condition for $R_{(K_n , \alpha)}$ over GCD domains by gathering Proposition~\ref{4prop2} and Theorem~\ref{4thm3}.

\begin{theorem} Let $\{ F_1 , F_2, \ldots , F_n \} \subset R_{(K_n , \alpha)}$  and $R$ be a GCD domain. Then $\{ F_1 , F_2 , \ldots , F_n \}$ forms a basis for $R_{(K_n , \alpha)}$ if and only if $\big\vert F_1 \squad F_2 \squad \ldots \squad F_n \big\vert = r \cdot Q_{K_n}$ where $r \in R$ is a unit.
\label{4thm4}
\end{theorem}

%%%%%%%%%%%%%%%%%%%%%%%%%%%%%%%%%%%%%%%%%%%%%%%%%%%%%%%%%
%%%%%%%%%% Basis Condition on Arbitrary Graphs Section
%%%%%%%%%%%%%%%%%%%%%%%%%%%%%%%%%%%%%%%%%%%%%%%%%%%%%%%%%

\section{Basis Condition on Arbitrary Graphs}\label{bcoag}

In this section, we extend Theorem~\ref{4thm4} to arbitrary graphs and obtain a general basis condition over GCD domains. In order to do this, we define the completion of an edge labeled graph. Given an edge labeled graph $(G,\alpha)$, we define the completion of it, denoted by $(K_G , \alpha^*)$, by adding the missing edges $uv$ and labeling them as $\alpha^* (uv) = 1$. If $uv \in E(G)$, then we set $\alpha^* (uv) = \alpha(uv)$.

The following proposition states that the completion does not affect the set of splines.

\begin{proposition} Let $(G,\alpha)$ be an edge labeled graph and $(K_G , \alpha^*)$ be the completion of it. Then $R_{(G,\alpha)} = R_{(K_G , \alpha^*)}$.
\label{5prop1}
\end{proposition}

\begin{proof} Let $F \in R_{(G,\alpha)}$ and $uv \in E(K_G)$. If $uv \in E(G)$, then $F(u) - F(v) \in \alpha(uv) = \alpha^* (uv)$. Otherwise, $F(u) - F(v) \in \alpha^* (uv) = 1$ and $F \in R_{(K_G , \alpha^*)}$. Conversely, if $F \in R_{(K_G , \alpha^*)}$, then $F \in R_{(G,\alpha)}$ since $(G,\alpha) \subset (K_G , \alpha^*)$. Thus, $R_{(G,\alpha)} = R_{(K_G , \alpha^*)}$.
\end{proof}

The relation between $Q_G$ and $Q_{K_G}$ is stated below.

\begin{proposition} Let $(G,\alpha)$ be an edge labeled graph and $(K_G , \alpha^*)$ be the completion of it. Then $Q_G = Q_{K_G}$.
\label{5prop2}
\end{proposition}

\begin{proof} Fix a vertex $v_i \in V(G)$ with $i > 1$. Since $G \subset K_G$, any zero trail $\tp ^{(v_i , 0)}$ on $(G, \alpha)$ is a zero trail on $(K_G , \alpha^*)$. Assume that a zero trail $\tp_j ^{(v_i , 0)}$ on $(K_G , \alpha^*)$ contain an edge $uv$ with $\alpha^* (uv) = 1$, then $\big( \tp ^{(v_i , 0)} \big) = 1$ and such a zero trail does not effect $\LL_i$. Hence, $\LL_i$ is the same on $(G,\alpha)$ and $(K_G , \alpha^*)$, and so $Q_G = Q_{K_G}$.
\end{proof}

Consider the completion of $(G, \alpha)$. By Theorem~\ref{4thm4}, $\{ F_1 , F_2 , \ldots , F_n \}$ forms a basis for $R_{(K_G , \alpha^*)}$ if and only if $\big\vert F_1 \squad F_2 \squad \ldots \squad F_n \big\vert = r \cdot Q_{K_G}$ where $r \in R$ is a unit. Since $R_{(G,\alpha)} = R_{(K_G , \alpha^*)}$ and $Q_G = Q_{K_G}$ by Proposition~\ref{5prop1} and Proposition~\ref{5prop2}, Theorem~\ref{4thm4} extends to arbitrary graphs as follows.

\begin{theorem} Let $(G, \alpha)$ be an arbitrary graph with $n$ vertices, $\{ F_1 , F_2, \ldots , F_n \} \subset R_{(G , \alpha)}$  and $R$ be a GCD domain. Then $\{ F_1 , F_2 , \ldots , F_n \}$ forms a basis for $R_{(G , \alpha)}$ if and only if $\big\vert F_1 \squad F_2 \squad \ldots \squad F_n \big\vert = r \cdot Q_G$ where $r \in R$ is a unit.
\label{5thm1}
\end{theorem}

%%%%%%%%%%%%%%%%%%%%%%%%%%%%%%%%%%%%%%%%%%%%%%%%%%%%%%%%%
%%%%%%%%%% Acknowledgements Section
%%%%%%%%%%%%%%%%%%%%%%%%%%%%%%%%%%%%%%%%%%%%%%%%%%%%%%%%%

\section*{Acknowledgements}

The second author is supported by the Scientific and Technological Research Council of Turkey, International Postdoctoral Research Fellowship Program for Turkish Citizens T\"{U}B\.{I}TAK-2219 (Project Number:1059B192201169). In addition, he thanks Smith College Department of Mathematical Sciences for their hospitality.

%%%%%%%%%%%%%%%%%%%%%%%%%%%%%%%%%%%%%%%%%%%%%%%%%%%%%%%%%
%%%%%%%%%% Bibliography
%%%%%%%%%%%%%%%%%%%%%%%%%%%%%%%%%%%%%%%%%%%%%%%%%%%%%%%%%

\bibliographystyle{alpha} % We choose the "plain" reference style
\bibliography{fs}{} % Entries are in the refs.bib file

\end{document}